\documentclass[letterpaper, 10 pt, conference]{ieeeconf}  

\IEEEoverridecommandlockouts                              
\usepackage{hyperref}
\usepackage{makeidx}         
\usepackage{graphicx}        
\usepackage[pdftex]{color}
\usepackage{multicol}        
\usepackage[bottom]{footmisc}

\usepackage{amsfonts,amsmath,euscript,mathbbol,mathtools,dsfont,mathtools,amssymb}

\usepackage{subfig}
\usepackage[normalem]{ulem}

\newtheorem{theorem}{\bf Theorem}[section]

\newtheorem{proposition}{\bf Proposition}[section]
\newtheorem{definition}{\bf Definition}[section]

\newtheorem{remark}{\bf Remark}[section]

\newcommand{\bmat}{\left[ \begin{matrix}}
\newcommand{\emat}{\end{matrix} \right]}

\newcommand{\Ebb}{{\mathbb E}\,}

\newcommand{\bb}{\mathbf b}

\newcommand{\by}{\mathbf y}

\newcommand{\bx}{\mathbf x}

\newcommand{\bA}{\mathbf A}

\newcommand{\bC}{\mathbf C}

\newcommand{\bK}{\mathbf K}
\newcommand{\bI}{\mathbf I}

\newcommand{\bP}{\mathbf P}

\newcommand{\bV}{\mathbf V}
\newcommand{\bW}{\mathbf W}
\newcommand{\bX}{\mathbf X}
\newcommand{\bY}{\mathbf Y}

\newcommand{\bmu}{\boldsymbol{\mu}}

\newcommand{\bSigma}{\boldsymbol{\Sigma}}

\newcommand{\bGamma}{\boldsymbol{\Gamma}}

\newcommand{\cB}{\mathcal B}
\newcommand{\cC}{\mathcal C}
\newcommand{\cF}{\mathcal F}
\newcommand{\cG}{\mathcal G}
\newcommand{\cH}{\mathcal H}
\newcommand{\cK}{\mathcal K}

\newcommand{\cN}{\mathcal N}
\newcommand{\cO}{\mathcal O}
\newcommand{\cP}{\mathcal P}

\newcommand{\cS}{\mathcal S}
\newcommand{\cV}{\mathcal V}
\newcommand{\cX}{\mathcal X}
\newcommand{\cY}{\mathcal Y}

\newcommand{\R}{\mathbb R}

\newcommand{\X}{\mathbb X}
\newcommand{\Y}{\mathbb Y}
\newcommand{\Pbb}{\mathbb P}

\graphicspath{{./figures/}}

\title{\LARGE \bf{On the Projective Geometry of Kalman Filter}}

\author{Francesca Paola Carli and Rodolphe Sepulchre
\thanks{Francesca Paola Carli is with the Department of Electrical Engineering and Computer Science, University of Li\`{e}ge, Belgium, 
and visiting the Department of Engineering, University of Cambridge, United Kingdom. She acknowledges support from the FNRS, Belgium.
        {\tt\small fpc23@cam.ac.uk }}%
\thanks{Rodolphe Sepulchre is with the Department of Engineering, University of Cambridge, United Kingdom, 
        {\tt\small r.sepulchre@eng.cam.ac.uk}}%
}

\begin{document}

\maketitle
\thispagestyle{empty}
\pagestyle{empty}

\begin{abstract}
Convergence of the Kalman filter is best analyzed by studying the contraction of the Riccati map in the space of positive definite (covariance) matrices. In this paper, we explore how this contraction property relates to a more fundamental non--expansiveness property of filtering maps in the space of probability distributions endowed with the Hilbert metric. 
This is viewed as a preliminary step towards improving the convergence analysis of filtering algorithms over general graphical models.
\end{abstract}

\section{Introduction}\label{sec:introduction}

This paper is about the asymptotic behavior of the Kalman filter \cite{KalmanBucy1961}.  
The Kalman--Bucy filter merges predictions from a trusted model of the dynamics of the system with incoming measurements 
in order to get an accurate, real--time estimate of the unknown internal state of the system. 
The estimation relies on the computation of a positive semidefinite matrix $\bP$, the covariance of the estimation error.  
The difference equation verified by $\bP$ is a discrete--time algebraic Riccati equation. 
Kalman showed that, for a linear time--invariant system, under detectability conditions, the Riccati equation converges to a fixed point, 
which is unique under certain stabilizability conditions (\cite{Kalman1963}, see also \cite{Jazwinski1970}). 
The classical convergence analysis requires several steps, 
showing that the error covariance is upper bounded, 
that, with zero initial value, it is monotone increasing, so that it admits a limit, 
and then proving that the corresponding filter is stable and that the limit is the same for all initial covariances. 

In \cite{Bougerol1993} Bougerol  proposed a more geometric  convergence analysis by  
showing that the discrete--time Riccati iteration is a contraction for the Riemannian metric associated to the cone of positive definite matrices.
Other authors elaborated along these lines (see e.g. \cite{LiveraniWojtkowski1994, Wojtkowski2007, LawsonLim2007, GaubertQu2014}), 
showing that the Riccati operator is a contraction with respect  to other metrics (e.g. Thompson's metric) 
and providing explicit formulas for the contraction coefficients. 

{In this paper, we 
seek to relate the convergence of the Kalman iteration, and, in particular, of the Riccati flow, 
to the contraction of the (projective)  Hilbert metric under the action of a nonlinear map on the space of positive 
{measurable} functions 
(as opposed to the action of the nonlinear Riccati operator on the space of positive definite matrices). 
The choice of Hilbert metric seems to be particularly sensible in this context since, thanks to its property of being invariant under scaling, 
it allows to  study the convergence of a nonlinear iteration via the analysis of a linear one. 
To this end, the Kalman iteration  is seen as a specialization for Gaussian distributions of filtering algorithms for general hidden Markov models (HMMs) 
and the 
observation is made that the underlying iteration of these general filtering algorithms 
never expands the Hilbert metric. } 
{
This approach is more general than the analysis of the Riccati iteration but at the price of a weaker result, since only non expansiveness of the Hilbert metric can be shown. 
The gap between non expansiveness and contraction is certainly a non trivial one in the infinite dimensional space of probability
distributions. 
{Using the Hilbert metric, convergence results have been proved in  \cite{AtarZeitouni1997}, \cite{LeGlandOudjane2004} (see also \cite{LeGlandMevel2000} for some results concerning HMMs with finite state space) 
where problems arising from non--compact state spaces or heavy tailed distributions have been considered. }
We envision that this approach can open the way to a geometric analysis of filtering algorithms on 
{general graphical models, e.g., of arbitrary topology. }
}

{The paper is organized as follows. Section \ref{sec:Hilbert_metric} and \ref{sec:Kalman} establish common notation by introducing the Hilbert metric and the Kalman filter iteration. }
In Section \ref{sec:contraction_filtering} we show that the nonlinear iteration underlying  filtering algorithms for general HMMs 
{does not expand} the Hilbert metric on the space of {positive measurable functions.   }
{In Section \ref{sec:contraction_Kaman} 
we show that the Kalman iteration can indeed be seen as a particularization for Gaussian distributions of forward filtering algorithms for general HMMs  
and as such does not expand the Hilbert metric on the space of {positive measurable functions} endowed with the Hilbert metric. 
Section \ref{sec:on_convergence_KF} discusses convergence.} 
Section \ref{sec:discussion} ends the paper.

\textbf{Notation.} Throughout the paper if $\cK$ is a cone, we denote by $\cK^+$ the interior of $\cK$.  
In particular we will denote by $\cP$ ($\cP^+$) the cone of positive semidefinite (definite) matrices while  
$\cF$ ($\cF^+$) will be used to denote the {cone of nonnegative (positive) measurable functions with respect to a suitable $\sigma$--algebra. }

\section{Hilbert metric}\label{sec:Hilbert_metric}

The Hilbert metric was introduced in \cite{Hilbert1895}. 
Birkhoff \cite{Birkhoff1957} (see also \cite{Bushell1973}) showed that strict positivity of a mapping implies contraction in the Hilbert metric,  
paving the way to many contraction--based results in the literature of positive operators. 
The Hilbert metric is defined as follows. 
Let $\cB$ be a real Banach space and let $\cK$ be a closed solid cone in $\cB$  that is a closed subset $\cK$ with the properties that 
(i) $\cK^+$ 
is non--empty;  
(ii) $\cK + \cK \subseteq \cK$; 
(iii) $\cK \cap - \cK = \left\{0\right\}$;  
(iv) $\lambda \cK \subset \cK$ for all $\lambda \geq 0$.  
Define the partial order
$$x\preceq y \Leftrightarrow y-x\in \cK\,,$$
and for $x, y \in \cK \backslash \left\{0\right\}$, let 
\begin{align*}
M(x, y) &:= \inf \left\{ \lambda | x -  \lambda y \preceq 0\right\}\\
m(x, y) &:= \sup \left\{\lambda |  x -  \lambda y \succeq 0 \right\}
\end{align*}
The Hilbert metric $d_{\cH}(\cdot, \cdot)$ induced by $\cK$ is defined by 
\begin{equation}\label{eqn:Hilbert_metric}
d_{\cH}\left(x,y\right) := \log\left(\frac{M(x,y)}{m(x,y)}\right), \; \;  x,y \in \cK \backslash \left\{0\right\}\,.
\end{equation}
For example, if $\cB = \R^n$ and the cone $\cK$ is the positive orthant, $\cK = \cO:=  \left\{ (x_1, \dots, x_n)\,:\, x_i \geq 0, \, 1 \leq i \leq n \right\}$, 
then $M(\bx,\by)=\max_{i}(x_i/y_j)$ and $m(\bx,\by)=\min_i(x_i/y_i)$ and the Hilbert metric can be expressed as 
$$
d_{\cH}(\bx,\by) = \log \frac{\max_i(x_i/y_i)}{\min_i{(x_i/y_i)}}
$$
On the other hand, if $\cB =\cS := \left\{ \bX = \bX^\top \in \R^{n \times n}\right\}$ is the set of symmetric matrices and $\cK = \cP:= \left\{ \bX \succeq 0 \mid \bX \in \cS\right\}$ is the cone of positive semidefinite matrices, then for $\bX,\bY \succ 0$, $M(\bX,\bY)=\lambda_{max}\left(\bX\bY^{-1}\right)$ and $m(\bX,\bY)= \lambda_{min}\left(\bX \bY^{-1}\right)$. Hence the Hilbert metric is 
$$
d_{\cH}(\bX,\bY) = \log \frac{\lambda_{max}\left(\bX \bY^{-1}\right)}{\lambda_{min}\left(\bX \bY^{-1}\right)}
$$
In the following, we will be interested to positive operators on finite measures. In this context, the Hilbert metric is defined as follows.  
Let  $\X$ be a complete separable metric space 
and let $\cX$ be the $\sigma$--algebra of Borel subsets of $\X$ . 
Moreover let $\cB = \cV$ be the vector space of finite signed measure on $(\X,\cX)$ and $\cK = \cC(\X)$  be the set of finite nonnegative measures on $\X$. 
We recall that two elements $\lambda, \mu \in \cC(\X)$ are called \emph{comparable} if  $\alpha \lambda \leq \mu \leq \beta \lambda$ for suitable positive scalars $\alpha, \beta$.  
The Hilbert metric on $\cC(\X) \backslash \left\{0\right\}$ is defined as  
$$
d_{\cH}(\mu,\mu') = 
\begin{cases}
\log \frac{ \sup_{A: \mu'(A)>0} \mu(A)/\mu'(A)}{ \inf_{A : \mu'(A)>0} \mu'(A)/\mu(A) } & \text{if} \,\mu, \mu' \, \text{comparable} \\
\infty & \text{otherwise.}
\end{cases}
$$

{ 
An important property of the Hilbert metric is the following. 
The Hilbert metric is a \emph{projective metric} on $\cK$ i.e. it is nonnegative, symmetric, it satisfies the triangle inequality and is such that, for every $x,y \in \cK$, 
$d_{\cH}(x,y)=0$ if and only if $x=\lambda y$ for  some $\lambda > 0$. It follows easily that $d_{\cH}(x,y)$ is constant on rays, that is 
\begin{equation}\label{eqn:Hilbert_metric_invariant_under_scaling}
d_{\cH}\left(\lambda x, \mu y\right) = d_{\cH}\left(x, y\right) \quad \text{for } \lambda, \mu > 0 \,. 
\end{equation}
}

\subsection*{Hilbert metric and positive mappings}

In this section, we review contraction properties of positive operators with respect to the Hilbert metric. 
{
We recall that a map $A: \cK \mapsto \cK$ is said to be \emph{positive}; 
a map $A: \cK^+ \mapsto \cK^+$ is said to be \emph{ strictly positive}.  
If $A$ is a strictly positive linear map we denote by 
\begin{equation}
k(A) := \inf \left\{  \lambda \, : d(Ax,Ay)  \leq \lambda d(x, y) \; \forall x,y, \in \cK^+ \right\}
\end{equation}
the contraction ratio of $A$ and by 
\begin{equation}
\Delta(A) := \sup \left\{ d(Ax,Ay)\, : \,  x,y, \in \cK^+ \right\}
\end{equation}
its projective diameter. 
Contraction properties of positive operators with respect to the Hilbert metric are established in the following theorem \cite{Birkhoff1957,Bushell1973,KohlbergPratt1982}. 
\begin{theorem}\label{prop:properties_Hilbert_metric} 
If  $x,y \in \cK$,  then the following holds 
\begin{itemize}
	\item[$(i)$] if $A$ is a positive linear map on $\cK$, then $d_{\cH}(Ax,Ay) \leq d_{\cH}(x,y)$,  
	i.e. the Hilbert metric contracts weakly under the action of a positive linear transformation. 	
	\item[$(ii)$] [Birkhoff, 1957] If $A$ is a strictly positive linear map in $\cB$, then  
	\begin{equation}
	k(A) = \rm{tanh} \frac14 \Delta(A)
	\end{equation}
\end{itemize}
\end{theorem}
\vspace{3mm}
Let $U$ denote the unit sphere in $\cB$ and let $E$ be the metric space $E:=\left\{ \cK^{+} \cup U, d_{\cH}\right\}$. 
Then, by combining Theorem \ref{prop:properties_Hilbert_metric} (ii), with the Banach contraction mapping theorem, the following generalization of the Perron--Frobenius  theorem holds: 
if $\Delta(A) < \infty$ and if the metric space $E$ is complete, then 
there exists a unique positive eigenvector of $A$ in $E$.

\section{Kalman filter and the Riccati operator}\label{sec:Kalman} 
 
In this section, we briefly introduce the Kalman filter iteration, that is analyzed later on in Section \ref{sec:contraction_Kaman} where an alternative derivation is also provided. 

Let us consider a linear dynamical system 
\begin{subequations}\label{eqn:LDS}
\begin{align}
\bX_{k+1}& = \bA\bX_k + \bW_k\, ,  \qquad k\geq 0 \label{eqn:state_eq}\\
\bY_k & = \bC\bX_k + \bV_k \,, \label{eqn:meas_eq}
\end{align}
\end{subequations}
where $\left\{\bW_k\right\}$ and $\left\{\bV_k\right\}$ are mutually uncorrelated white noise Gaussian processes with variance $\bGamma$ and  $\bSigma$, respectively, 
i.e. 
\begin{equation}\label{eqn:LDS_noises}
\bW_k  \sim \cN\left(0, \bGamma\right) 
\; \bV_k  \sim \cN\left(0, \bSigma\right),   
\end{equation}  
and with initial condition 
\begin{equation}\label{eqn:LDS_X0}
\bX_{0} \sim \cN(\bmu_{0},\bP_{0}) 
\end{equation}
such that 
\begin{equation}\label{eqn:LDS_covX0-noises}
\Ebb\left[ \bW_{k} \bX_{0}^{\top}\right] = 0, \qquad \Ebb\left[ \bV_{k} \bX_{0}^{\top}\right] = 0\,. 
\end{equation}
The Kalman filter recursion consists of the following steps:

\noindent \textbf{Time update (``Predict'')  step: }
\begin{align}
\hat \bX_{k|k-1} &=  \bA \hat \bX_{k-1|k-1}\label{eqn:kalman_predict_state}\\
\bP_{k|k-1} & = \bA \bP_{k-1|k-1} \bA^\top + \bGamma \label{eqn:kalman_predict_covariance}
\end{align}

\noindent \textbf{Measurement update (``Correct'') step: }
\begin{align}
\hat \bX_{k | k}& 
=  \hat \bX_{k|k-1} + \bK_k\left(\bY_k - \bC \hat \bX_{k|k-1} \right)  \label{eqn:kalman_update_state}\\
\bP_{k | k}& = \left(\bI-\bK_k\bC\right)\bP_{k|k-1}  \label{eqn:kalman_update_covariance}\\
\bK_k & =  \bP_{k|k-1} \bC^\top \left(\bC \bP_{k|k-1} \bC^\top + \bSigma\right)^{-1} \label{eqn:kalman_update_gain}
\end{align}
and is initialized at $\hat \bX_{0|-1} =  \bmu_0$, $\bP_{0|-1} = \bP_0$. 
Equivalently, the following one--step expression for the a posteriori state estimate and covariance holds 
\begin{align}
\bP_{k | k}&  = \Phi(\bP_{k-1|k-1})\\
\hat \bX_{k | k}& =  \left(\bA - \bP_{k | k}\bC^\top \bSigma^{-1} \bC \bA\right) \hat \bX_{k-1|k-1}\nonumber  \\
& \qquad \qquad \qquad \qquad  \qquad \qquad + \bP_{k|k }\bC^\top \bSigma^{-1} \bY_k
\end{align}
where $\Phi$ is the nonlinear map 
\begin{align}
\Phi(\bP)& = \left(\bA \bP \bA^\top + \bGamma\right) \nonumber \\
& \big[\bI + \bC^\top \bSigma^{-1} \bC \bGamma  + \bC^\top \bSigma^{-1} \bC\bA \bP \bA^\top\big]^{-1}\,.  \label{eqn:DARE_1}
\end{align}
$\Phi$ in \eqref{eqn:DARE_1} can be written as 
\begin{equation}\label{eqn:DARE_2}
\Phi (\bP) = \left( \left(\bA \bP \bA^\top + \bGamma\right)^{-1} +  \bC^\top \bSigma^{-1} \bC \right)^{-1}\,.
\end{equation}
This equation is called the \emph{discrete Riccati equation}.  
{
In the literature, convergence of the Kalman iteration has been studied by proving that the discrete Riccati operator contracts suitable metrics (e.g. the Riemannian metric \cite{Bougerol1993}, 
the Thompson's part metric \cite{LiveraniWojtkowski1994}) on the set of positive definite matrices. 
In the following, we propose to study convergence of the Kalman iteration by directly analyzing an equivalent iteration on the space of positive measurable functions. 
This equivalent iteration will be introduced and discussed in the following section.}

\section{
{Non--expansiveness of the Filtering Recursion in Projective Spaces}}
\label{sec:contraction_filtering}

In this section, we introduce the filtering algorithm for general hidden Markov models and we show that the map underlying the main iteration 
{does not expand} the Hilbert metric on the {cone of positive measurable functions. }
Note that some authors use the term hidden Markov model exclusively for the case where $\bX_k$ takes values in a finite state space.  
In this paper, following e.g. \cite{CappeMoulinesRyden2005}, 
when referring to a hidden Markov model we also intend to include models with continuous state space; 
such models are also referred to as state--space models in the literature.

\subsection*{Problem statement}

In the broadest sense of the word, a hidden Markov model is a Markov process that is split into two components: an observable component and an unobservable or ``hidden" component. 
That is, a hidden Markov model is a Markov process $\left\{\bX_k, \bY_k\right\}_{k \geq 0}$ on the state space $\X \times \Y$, 
where we presume that we have a way of observing $\bY_k$, but not $\bX_k$.

In simple cases such as discrete--time, countable state space models, it is common to define hidden Markov models by using the concept of conditional independence.  
It turns out that conditional independence is mathematically more difficult to define in general settings 
(in particular, when the state space $\X$ of the Markov process is not countable -- the case we are interested in), 
so a different route is adopted (see \cite{CappeMoulinesRyden2005} for details).    
To this aim, we 
define the transition kernel (the parallel of the transition matrix for countable state spaces).  
\begin{definition}\textbf{(Transition kernel)} 
A \emph{kernel} from a measurable space $(\X,\cX)$ to a measurable space $(\Y,\cY)$ 
is a map $Q : \X \times \cY \rightarrow [0,\infty]$ such that 
\\(i) for all $\bx \in \X$, $A \mapsto Q(\bx, A)$ is a measure on $\Y$; 
\\(ii) for all $A\in \cY$, the map  $\bx\mapsto Q(\bx,A)$ is measurable.
\\If $Q(\bx,\Y) = 1$ for every $\bx \in \X$, then $Q$ is called a \emph{transition kernel}. 
\end{definition}
We next consider an $\X$--valued stochastic process $\left\{\bX_k\right\}_{k \geq0}$, i.e., a collection of $\X$--valued random variables  
on a common underlying probability space $(\Omega,\cG,\Pbb)$, where $\X$ is some measure space. 
{
The process $\left\{\bX_k\right\}_{k \geq0}$ is \emph{Markov} if, for every time $k\geq 0$, there exists a transition kernel $Q_{k}: \X \times \cX \rightarrow [0,1]$  such that
$$
\Pbb(\bX_{k+1} \in A\mid \bX_0, .\dots ,\bX_k) = Q_{k}(\bX_k,A) \,,
$$
for every   $A \in \cX$, $k \geq 0$. If $Q_{k}=Q$ for every $k$, then the Markov process is called \emph{homogeneous}. 
For simplicity of exposition,  
from now on we will 
consider homogeneous Markov processes, though the theory we are about to develop does not rely on this assumption. }
A \emph{hidden Markov model} $\left\{\bX_k, \bY_k\right\}_{k \geq 0}$ is  a (only partially observed) Markov process, 
whose transition kernel has a special structure, namely it is such that both the joint process $\left\{\bX_k, \bY_k\right\}_{k \geq 0}$ 
and the marginal unobservable process $\left\{\bX_k\right\}_{k \geq 0}$ are Markov. 
Formally: 

\begin{definition} \textbf{(Hidden Markov Model)} Let $(\X,\cX)$ and $(\Y,\cY)$ be two measurable spaces and 
let $Q$ and $G$ denote a transition kernel on $(\X,\cX)$ and a transition kernel from $(\X,\cX)$ to $(\Y,\cY)$. 
Consider the transition kernel on the product space $(\X \times \Y, \cX \otimes \cY)$ defined by 
$$
T[(\bx,\by),C] = \iint_C Q(\bx,d\bx')G(\bx',d\by')\,. 
$$
for  $(\bx,\by) \in \X \times \Y, \,C \in \cX \otimes \cY$.  
The Markov process  $\left\{\bX_k,\bY_k\right\}_{k \geq 0}$ with transition kernel $T$ and initial  probability measure $\mu$ on $(\X,\cX)$, is called a \emph{hidden Markov model}. 
\end{definition} 
A hidden Markov model is completely determined by the initial measure $\mu$ and its transition kernel $T$ (equivalently by $Q$ and $G$),  
formally:  
\begin{proposition} 
Let $\left\{ \bX_k,\bY_k \right\}_{k \geq 0}$ be a hidden Markov model on $(\X \times \Y, \cX \otimes \cY)$ 
with transition kernel $Q$, observation kernel $G$, and initial measure $\mu$. Then
for every bounded measurable function $f : \X \times \Y \rightarrow \R$, 
\begin{align}\label{eqn:factorization_prop_HMM}
\Ebb [ &f(\bX_0, \bY_0,  \dots , \bX_k, \bY_k) ] \nonumber \\
 = \int & f(\bx_0, \by_0, . . . , \bx_k, \by_k) G(\bx_k, d\by_k) Q(\bx_{k-1}, d\bx_k) \dots  \nonumber\\
&G(\bx_1, d\by_1) Q(\bx_0, d\bx_1) G(\bx_0, d\by_0) \mu(d\bx_0).	
\end{align}
\end{proposition} 
\vspace{2mm}
In the following, we are interested in the  \emph{filtering problem} for HMM, namely the problem of computing the sequence of conditional distribution of $\bX_{k}$ given $\bY_{0:k}$. 
The filtering, as well as the related smoothing and prediction problems, have their origin in the work of Wiener, who was interested in stationary processes. 
In the more general setting of hidden Markov models, early contributions are the works of 
Stratonovich, Shiryaev, Baum, Petrie and coworkers \cite{Stratonovich1960,Shiryaev1966, BaumPetrieSoulesWeiss1970}, see also \cite{CappeMoulinesRyden2005} for a recent monograph.

\subsection*{Filtering algorithm} 

Assume that both $G$ and $Q$ are absolutely continuous with respect to the Lebesgue measure 
(in the next section we will particularize to the case of Gaussian distributions)  
with transition density functions $g$ and $q$ respectively. 
In terms of transition densities, the filtering problem can be solved as follows.   

\begin{theorem}[\textbf{Forward filtering recursion}]\label{thm:normalized_filtering_recursion}
We denote by $\hat{\alpha}_k(\bx_k)$ the probability density  function 
$$
\hat{\alpha}_s(\bx_k)  := p(\bx_k|\by_{0:s}) \\
$$
and let 
$$
g(\bx_k,\by_k)  = g_k(\bx_k)\,. 
$$
Then $\hat \alpha_k(\bx_k) = p(\bx_k\mid \by_{0:k})$ can be recursively expressed in terms of $\hat \alpha_{k-1}(\bx_{k-1}) = p(\bx_{k-1}\mid \by_{0:k-1})$ as follows 
\begin{equation}\label{eqn:recursion_alpha_hat}
\hat{\alpha}_k(\bx_{k})   = \frac{g_k(\bx_k) \int q(\bx_{k-1},\bx_k)  \hat{\alpha}_{k-1}(\bx_{k-1})d{\bx_{k-1}}}{\iint g_k(\bx_{k}) q(\bx_{k-1},\bx_k)  \hat{\alpha}_{k-1}(\bx_{k-1})d{\bx_k} d{\bx_{k-1}}}\,  
\end{equation}
with iteration initialized at 
\begin{equation}\label{eqn:initialization_filtering_alpha_hat}
\hat \alpha_0(\bx_0) 
= \frac{g_0(\bx_0) \mu(\bx_0)}{\int g_0(\bx_0) \mu(\bx_0)d\bx_0}. 
\end{equation}
\end{theorem}
\vspace{2mm}
{The iteration \eqref{eqn:recursion_alpha_hat} defines a  
time--varying 
dynamical system 
over the {cone $\cF$ of nonnegative measurable functions  with respect to the product $\sigma$--algebra $\cX \otimes \cY^{\otimes(k+1)}$. } 
The following equivalent two--step formulation holds. 
\begin{remark}\textbf{[Two--step formulation of the filtering recursion]} 
The filtering recursion \eqref{eqn:recursion_alpha_hat} is often split into two steps.  
\begin{enumerate}
	\item \textbf{prediction step:} in which the one-step-ahead predictive density is computed 
   \begin{equation}\label{eqn:forward_algo_SSM_predict_step}
   \hat{\alpha}_{k-1}(\bx_k) 
  =   \int q(\bx_{k-1},\bx_k ) \hat{\alpha}_{k-1}(\bx_{k-1})  d{\bx_{k-1}} 
  \end{equation}
  \item \textbf{update step:} in which the observed data from time $k$ is absorbed yielding to the filtering density 
  \begin{equation}\label{eqn:forward_algo_SSM_update_step}
   \hspace{-2mm} \hat{\alpha}_k(\bx_k) 
    =  \frac{g_k(\bx_k) \hat{\alpha}_{k-1}(\bx_k)}{\iint  g_k(\bx_{k}) q(\bx_{k-1},\bx_k)  \hat{\alpha}_{k-1}(\bx_{k-1})d{\bx_k} d{\bx_{k-1}}} 
  \end{equation}
  \end{enumerate}
\end{remark}

\subsection*{{Non--expansiveness in projective space} }

{First of all, notice that the nonlinear map in \eqref{eqn:recursion_alpha_hat}, say $\bar{\Psi}_{k}$, is the composition of a linear one (at the numerator) and a positive scaling, i.e. we can write
\begin{equation*}
(\bar{\Psi}_k  f)(\bx) = \frac{(\Psi_k  f)(\bx)}{\int (\bar{\Psi}_k  f)(\bx) d\bx}
\end{equation*}
where 
\begin{equation}\label{eqn:Psi_k}
(\Psi_k  f)(\bx) = g_k(\bx) \int q(\bx',\bx)  f(\bx')d{\bx'} 
\end{equation}
with $q$ and $g$ transition densities associated to the transition and observation kernels $Q$ and $G$, respectively. 
The next theorem draws the consequences of the fact that the map  $\Psi_{k}$ takes nonnegative measurable functions into nonnegative measurable functions.  }

\begin{theorem}\label{thm:contraction_filtering_HMM} 
{The map $\Psi_{k}$ in \eqref{eqn:Psi_k} does not expand the Hilbert metric, i.e. $$d_{\cH}((\Psi_{k} f) (\bx), (\Psi_{k} g)(\bx)) \leq d_{\cH}(f(\bx),g(\bx))\,.$$}
\end{theorem}
\vspace{2mm}
\begin{proof} 
{The map $\Psi_{k}$ is the composition of  
(i) $(\Psi^{(1)} f)(\bx) = \int q(\bx',\bx)  f(\bx')d{\bx'} $ and 
(ii) $(\Psi^{(2)} f)(\bx)= g_k(\bx) f(\bx)$. 
The maps $\Psi^{(1)}$ and $\Psi^{(2)}$ are positive linear and as such they do not expand the Hilbert metric (see Theorem \ref{prop:properties_Hilbert_metric}, (i)). 
The thesis follows since the composition of nonexpansive operators is nonexpansive. }
\end{proof}

\section{Kalman filtering as Forward Filtering Recursion}\label{sec:contraction_Kaman} 

The classical derivation of Kalman filter relies on an argument based on projections onto spaces spanned by random variables. 
As an alternative, the Kalman iteration can be seen as a  specialization of the filtering algorithm in Theorem \ref{thm:normalized_filtering_recursion} for Gaussian distributions. 
This fact by itself is known in the literature (see e.g. \cite{CappeMoulinesRyden2005}).  
{In this section, first we briefly review this
alternative derivation of Kalman filtering. 
This, combined with the (weak) contraction result of Theorem \ref{thm:contraction_filtering_HMM}, let us conclude that the Kalman iteration does not expand the Hilbert metric. 
Convergence of the Kalman iteration is discussed in Section \ref{sec:on_convergence_KF}.}

Before getting started, we observe that the linear dynamical system  \eqref{eqn:LDS}--\eqref{eqn:LDS_covX0-noises} 
is indeed equivalent to a hidden Markov model as specified by \eqref{eqn:factorization_prop_HMM} 
with initial, transition and emission probability densities, for $k \geq 0$, given by 
\begin{align}\label{eqn:initial_kransition_emission_prob_LSD}
p(\bx_0) & = \cN\left(\bmu_0,\bP_0\right) , \\
p(\bx_{k+1} \mid \bx_{k})  &= \cN\left(\bA \bx_{k},\bGamma\right),\\
p(\by_k \mid \bx_{k})&  = \cN\left(\bC \bx_{k},\bSigma\right), 
\end{align} 
Also we recall that given the prior and likelihood 
\begin{align}
p(\bx) & = \cN(\bmu_X, \bSigma_{X}) \\
p(\by \mid \bx) & = \cN(\bA \bx + \bb, \bSigma_{Y|X}) 
\end{align}
the posterior $p(\bx \mid \by)$ and normalization constant $p(\by)$ are given by 
\begin{align}
p(\by) &  = \cN(\bA \bmu_X + \bb,\bSigma_{Y \mid X} + \bA \bSigma_X \bA^\top) \label{eqn:p_y} \\
p(\bx \mid \by) & = \cN\left( \bmu_{X\mid Y}, \bSigma_{X \mid Y}\right)
\label{eqn:p_x_given_y}
\end{align}
with 
\begin{align}\label{eqn:posterior_mean_and_covariance}
 \bSigma_{X \mid Y} & = \bSigma_X+\bA^\top \bSigma_{Y|X}^{-1}\bA \\
\bmu_{X\mid Y} &=  \bSigma_{X \mid Y} \left[\bA^\top \bSigma_{Y \mid X}^{-1}(\by-\bb) + \bSigma_X^{-1}\bmu_X\right]\,. 
\end{align}

The next proposition connects the Kalman filter algorithm to the filtering recursion described in Section \ref{sec:contraction_filtering}. 

\begin{proposition}\label{prop:Kalman_as_ForwardAlgorithm}
The Kalman filter recursion \eqref{eqn:kalman_predict_state}--\eqref{eqn:kalman_update_gain} 
is a specialization of the forward filtering recursion of Theorem \ref{thm:normalized_filtering_recursion} 
for an HMM with Gaussian initial, transition and emission probabilities 
as in \eqref{eqn:initial_kransition_emission_prob_LSD}. 
\end{proposition}
\begin{proof}
Let 
\begin{align*}
\bmu_{k\mid s} &:= \Ebb \left[\bX_k \mid \bY_{0:s}\right], \\
\bP_{k \mid s} &: = \Ebb \left[(\bX_k-\bmu_{k \mid s})(\bX_k-\bmu_{k\mid s})^\top \mid\bY_{0:s} \right]
\end{align*}
\begin{enumerate}
	\item \textbf{prediction step:} By \eqref{eqn:forward_algo_SSM_predict_step}, $p(\bx_k|\by_{0:k-1})$ is given by 
  $$ 
  p(\bx_k |\by_{0:k-1}) =  \int_{\bx_{k-1}} {p(\bx_k |\bx_{k-1} )}  
	{p(\bx_{k-1} | \by_{0:k-1})} 
	d\bx_{k-1} 
  $$
	Now, $p(\bx_k |\bx_{k-1} )$ is Gaussian with mean $\bA \bx_{k-1}$ and covariance $\bGamma$.   
$p(\bx_{k-1} | \by_{0:k-1})$ is also Gaussian. We denote by $\bmu_{k-1\mid k-1}$  and $\bP_{k-1\mid k-1}$ its mean and covariance, respectively. 
By virtue of \eqref{eqn:p_y} we get 
$$
{p(\bx_k |\by_{0:k-1})} \sim \cN(\bA \bmu_{k-1|k-1}, \bA \bP_{k-1|k-1} \bA^\top + \bGamma)
$$
i.e. 
\begin{align*}
\bmu_{k\mid k-1} & = \bA \bmu_{k-1|k-1}\\
\bP_{k\mid k-1} & = \bA \bP_{k-1|k-1} \bA^\top + \bGamma
\end{align*}
which are the a priori state estimate and covariance in \eqref{eqn:kalman_predict_state}--\eqref{eqn:kalman_predict_covariance}.   \item \textbf{update step:} By \eqref{eqn:forward_algo_SSM_update_step}, $p(\bx_k |\by_{0:k})$ is given by 
	$$
 {p(\bx_k |\by_{0:k})} = 
	\frac{{p(\by_k \mid \bx_k )} 
	{{p(\bx_k |\by_{0:k-1})}} 
	}	{p(\by_k \mid \by_{0:k-1})}
  $$
	Now $p(\by_k \mid \bx_k )$ is Gaussian with mean $\bC \bx_k$ and covariance $\bSigma$.  
	$p(\bx_k |\by_{0:k-1})$ is also Gaussian. We denote by $\bmu_{k \mid k-1}$ and $\bP_{k \mid k-1}$ its mean and covariance.  
	By virtue of \eqref{eqn:p_x_given_y} we get 
	$$
	{p(\bx_k |\by_{0:k})} \sim \cN \left(\bmu_{k\mid k}, \bP_{k\mid k} \right)
	$$
	with 
	\begin{align}
	\bP_{k\mid k} & = \left(\bP_{k \mid k-1}^{-1} + \bC^\top \bSigma^{-1} \bC\right)^{-1} \label{eqn:Ptt}\\
	\bmu_{k\mid k}& = \bP_{k\mid k} \left[\bC^\top \bSigma^{-1} \by_k + \bP_{k \mid k-1}^{-1} \bmu_{k\mid k-1}\right] \label{eqn:mutt}
	\end{align}
	from which the expressions \eqref{eqn:kalman_update_state}--\eqref{eqn:kalman_update_covariance} for the a posteriori state estimate and covariance 
	can be recovered via  the matrix inversion lemma.
	\end{enumerate}
\end{proof}

By the results in Theorem \ref{thm:contraction_filtering_HMM} and Proposition \ref{prop:Kalman_as_ForwardAlgorithm}, 
we have that the map underlying the Kalman filtering algorithm does not expand the Hilbert metric on 
{space of positive measurable functions.}

\section{On strict contractiveness of the Kalman iteration}\label{sec:on_convergence_KF}

So far, we have shown that the time--varying nonlinear operator $\bar{\Psi}_{k}$ that underlies the Kalman iteration does not expand the Hilbert metric. 
Proving convergence of the Kalman iteration indeed amounts to prove that such iteration \emph{strictly} contracts the Hilbert metric. 
{As observed in Section \ref{sec:contraction_filtering},} 
the map \eqref{eqn:recursion_alpha_hat} is the composition of a linear positive map and a positive scaling. 
By the scaling invariant property of the Hilbert metric, it follows that 
convergence analysis can concentrate only on the linear numerator of $\bar{\Psi}_{k}$. 
By Theorem \ref{prop:properties_Hilbert_metric} (ii), a sufficient condition for a strictly positive linear operator to be a contraction 
is to have a finite projective diameter. 
At this point, one may observe that even the Hilbert distance between two Gaussians with the same variance and different mean may tend to infinity
{(a general discussion that takes into account problems arising from the use of the Hilbert metric with non--compact state space and heavy tailed distributions is contained in \cite{AtarZeitouni1997}}). 
Proving strict contraction usually requires to exploit that the map $\bar{\Psi}_{k}$ is time--varying,  
and showing that the map contracts \emph{over a uniform time--horizon} as opposed to at each time instant. 
{For iterations on the finite dimensional space of covariance matrices, this is the place where the observability and controllability conditions enter the analysis. }
Our hope is that similar conditions apply to more general situations that the one covered by the Kalman filter and that this general approach will 
find novel applications in the analysis of filtering algorithms on 
{general} graphical models.


\section{Conclusion}\label{sec:discussion}

As an attempt to generalize the contraction--based convergence analysis of the Kalman filter, we have interpreted the contraction result of  Bougerol in the space of positive definite (covariance) matrices as a specialization of the 
non--expansiveness of the general filtering recursion for hidden Markov models in the space of  positive measurable functions. 
In spite of the obstacles to showing a finite projective diameter in this infinite dimensional space, we feel that this approach is worth revisiting in the convergence analysis of filtering algorithms 
on 
{general}  
graphical models 
({arbitrary topology} and/or on different spaces of distributions).  
This is the topic of ongoing research.



\bibliographystyle{plain}
\bibliography{biblio_Kalman}

\end{document}